\documentclass[11pt,a4paper,leqno]{amsart}
\usepackage[T1]{fontenc}
\usepackage[utf8]{inputenc}
\usepackage{lmodern}
\usepackage{microtype}

\usepackage{CJK}
\usepackage{mathrsfs}
\usepackage{amsmath}

\usepackage{amsfonts}
\usepackage{geometry}
\usepackage{amsthm}
\usepackage{amssymb}
\usepackage{amscd}
\usepackage{constants}
\usepackage{latexsym}
\usepackage[final]{hyperref}
\usepackage[abbrev]{amsrefs}
\allowdisplaybreaks
\numberwithin{equation}{section}
\newtheorem{theorem}{Theorem}[section]

\newtheorem{lemma}[theorem]{Lemma}
\newtheorem{proposition}[theorem]{Proposition}
\theoremstyle{definition}

\numberwithin{equation}{section}
\newcommand{\abs}[1]{\lvert#1\rvert}
\newcommand{\norm}[1]{\lVert#1\rVert}
\newcommand{\st}{\,\vert\,}
\newcommand{\dif}{\,\mathrm{d}}

\newcommand{\Rset}{\mathbb{R}}
\newcommand{\Nset}{\mathbb{N}}
\newcommand{\weakto}{\rightharpoonup}
\DeclareMathOperator{\supp}{supp}

\title[Nonlinear Choquard equations]
      {Standing waves with a critical frequency for nonlinear Choquard equations }

\author[J. Van Schaftingen]{Jean Van Schaftingen}
\address{Jean Van Schaftingen\\
Institut de Recherche en Math\'{e}matique et Physique\\
Universit\'{e} Catholique de Louvain\\
Chemin du Cyclotron 2 bte L7.01.01\\
1348 Louvain-la-Neuve\\
Belgium}
\email{Jean.VanSchaftingen@uclouvain.be}

\author[J. Xia]{Jiankang Xia  (夏健康)}
\address{Xia Jiankang\\
Chern Institute of Mathematics and LPMC\\
Nankai University\\
Tianjin, 300071\\
China}
\email{fyxt001@163.com}
\thanks{J.\thinspace Van Schaftingen was supported by the Projet de Recherche (Fonds de la Recherche Scientifique--FNRS) T.1110.14 ``Existence and asymptotic behavior of solutions to systems of semilinear elliptic partial differential equations''.
Jiankang Xia is partially  supported by NSF of China (NSFC-11271201) and he acknowledges the support of the China Scholarship Council and the hospitality the Universit\'e catholique de Louvain (Institut de Recherche en Math\'ematique et en Physique).}

\begin{document}
\begin{abstract}
In this paper, we study the nonlocal Choquard equation
\[
-\varepsilon^2 \Delta u_\varepsilon + V u_\varepsilon=\bigl(I_\alpha \ast \abs{u_\varepsilon}^p\bigr)\abs{u_\varepsilon}^{p-2}u_\varepsilon
\]
where $N\geq 1$, $I_\alpha$ is the Riesz potential of order \(\alpha \in (0, N)\) and $\varepsilon>0$ is a parameter. When the nonnegative potential $V\in C (\Rset^N)$ achieves \(0\) with a homogeneous behaviour or on the closure of an open set but remains bounded away from \(0\) at infinity, we show the existence of groundstate solutions for small $\varepsilon>0$ and exhibit the concentration behaviour as $\varepsilon\to 0$.
\end{abstract}
\begin{CJK}{UTF8}{gbsn}
\maketitle
\end{CJK}

\tableofcontents
\section{Introduction and main results}
\label{section1}
We are interested in the following nonlinear \emph{Choquard equation}
\begin{equation}
\tag{$\mathcal{C}_\varepsilon$}
\label{eqChoquard}
-\varepsilon^2 \Delta u_\varepsilon + V u_\varepsilon=\bigl(I_\alpha \ast \abs{u_\varepsilon}^p\bigr)\abs{u_\varepsilon}^{p-2 }u_\varepsilon  \qquad \text{ in } \mathbb{R}^N
\end{equation}
where the dimension $N\in\mathbb{N}_*=\{1,2,\ldots\}$ of the \emph{Euclidean space} \(\Rset^N\) is given and $V\in C(\mathbb{R}^N,[0,+\infty))$ is an \emph{external potential}. 
The function $I_\alpha:\mathbb{R}^N\backslash{\{0\}}\to \mathbb{R}$ is the \emph{Riesz potential} of order $\alpha\in(0,N)$, defined for each \(x \in \mathbb{R} \setminus \{0\}\) by
\[
  I_{\alpha}(x)=\frac{\Gamma(\frac{N-\alpha}{2})}{2^\alpha\Gamma(\frac{\alpha}{2})\pi^{\frac{N}{2}}\abs{x}^{N-\alpha}},
\]
where $\Gamma$ is the classical Gamma function, and $\varepsilon>0$ is a small parameter.

The nonlocal semilinear equation \eqref{eqChoquard} with $N=3$, $\alpha=2$ and $p=2$ is known as the Choquard--Pekar equation. It appears in several physical contexts: standing waves for the Hartree equation, Pekar's quantum physical model of a polaron at rest \cite{P}, Choquard's model of an electron trapped in its own hole \cite{L} and a model  coupling the Schr\"odinger equation of quantum mechanics and the classical Newtonian gravitational potential \cites{Diosi1984,J1,J2,MPT,Penrose1996}.

The existence and qualitative properties of solutions of the Choquard equation \eqref{eqChoquard} have been studied mathematically for a few decades when $\varepsilon$ is a fixed constant by variational methods \cites{L,Lions1980,Lions1984,MVJFA} (see also the review \cite{MVSReview} and the references therein). In quantum physical models, the parameter $\varepsilon$ is an adimensionalized Planck constant which in the \emph{semiclassical limit r\'egime} is quite small. In general, one expects to recover some classical dynamics in this r\'egime.

This semiclassical limit is well understood for the nonlinear Schr\"odinger equation
\begin{equation*}
 -\varepsilon^2 \Delta u_\varepsilon + V u_\varepsilon= \abs{u_\varepsilon}^{q-2 }u_\varepsilon  \qquad \text{ in } \mathbb{R}^N
\end{equation*}
Under the assumption that \(\inf_{\Rset^N} V > 0\), solutions concentrating at critical points of the potential \(V\) have been construted by topological and variational methods \citelist{\cite{FloerWeinstein1986}\cite{OH}\cite{Rabinowitz1992}\cite{ABC}\cite{AM2006}\cite{AM2007}\cite{DF1997}\cite{wang}}.
The remaining case \(\inf_{\Rset^N} V = 0\) corresponds to the \emph{critical frequency}. When \(\inf_{\Rset^N} V = 0\) and \(V > 0\) on \(\Rset^N\), such constructions are still possible provided the function \(V\) does not decay too fast at infinity or \(q\) is large enough \citelist{\cite{BonheureVanSchaftingen2008}\cite{MorozVanSchaftingen2010}\cite{AmbrosettiMalchiodiRuiz2006}}. When the potential \(V\) vanishes somewhere in \(\Rset^N\), then the solutions exhibit a different concentration behaviour which was studied by J. Byeon and Z.-Q. Wang \citelist{\cite{ByeonWang2002}\cite{ByeonWang2003}}.

For the Choquard equation \eqref{eqChoquard}, the semiclassical limit has been studied in the subcritical frequency case \(\inf_{\mathbb{R}^N} V > 0\) \citelist{\cite{WeiWinter2009}\cite{CingolaniSecchiSquassina2010}} (with extensions to the quasilinear case \citelist{\cite{AYJDE}\cite{AYJMP}} and to general nonlinearities \cite{YZZ}) and when \(\inf V = 0\) and \(V > 0\) \citelist{\cite{MorozVanSchaftingen2015}\cite{Secchi2010}}.

\medbreak

In this work we study a large class of potential \(V\) that vanishes somewhere on \(\Rset^N\).
\begin{theorem}\label{theoremGammaHomogeneous}
Let $N\in \Nset_*$, $p \in (0, +\infty)$ and \(V \in C (\Rset^N)\).
If
\[
 \frac{N-2}{N+\alpha} < \frac{1}{p} < \frac{N}{N+\alpha},
\]
if \(V \ge 0\) on \(\Rset^N\), if
\[
 \varliminf_{\abs{x} \to \infty} V (x) > 0,
\]
and if there exists \(\gamma > 0\) such that,
for every \(x \in \Rset^N\) either
\[
 \lim_{z \to x} \frac{V (z)}{\abs{z - x}^\gamma} = + \infty
\]
or there exist a positive \(\gamma\)--homogeneous function \(W \in C (\Rset^N)\) such that
\[
 \lim_{z \to x} \frac{V (z) - W (z - x)}{\abs{z - x}^{\gamma}} = 0,
\]
and if there exists at least one point \(x \in \Rset^N\) such that the second alternative holds,
then, for sufficiently small $\varepsilon>0$, equation~\eqref{eqChoquard} has a positive groundstate solution $u_\varepsilon$. \\
Moreover, there exist \(x_* \in \Rset^N\), a positive \(\gamma\)--homogeneous function \(W_* \in C (\Rset^N)\) such that
\[
 \lim_{z \to x_*} \frac{V (z) - W_* (z - x_*)}{\abs{z - x_*}^{\gamma}} = 0,
\]
a groundstate \(v_* \in H^1_{W_*}(\Rset^N) \setminus \{0\}\) of the problem
\begin{align*}
 -\Delta v_* + W_* v_* &= \bigl(I_\alpha \ast \abs{v_*}^p\bigr) \abs{v_*}^{p - 2} v_* && \text{in \(\Rset^N\)},
\end{align*}
and a sequence \((\varepsilon_n)_{n \in \Nset}\) in \((0, +\infty)\) converging to \(0\) such that, as \(n \to \infty\),
\begin{align*}
  \varepsilon_n^{\frac{\alpha-\gamma}{(p-1)(\gamma+2)}}u_{\varepsilon_n}(x_*+\varepsilon_n^{\frac{2}{\gamma+2}}\cdot)&\to  v_* 
  &&\text{ in } H^1_{\mathrm{loc}}(\mathbb{R}^N),
\end{align*}
and
\[
\begin{split}
 \int_{\Rset^N} \abs{\nabla v_*}^2 + W_* \abs{v_*}^2
  =& \lim_{n \to \infty} \frac{1}{\varepsilon_n^{\frac{2}{\gamma + 2} (N + \frac{p \gamma - \alpha}{p - 1})}} \int_{\Rset^N} \varepsilon_n^2 \abs{\nabla u_{\varepsilon_n}}^2 + V \abs{u_{\varepsilon_n}}^2\\
  =&  \inf \,\Bigl\{\int_{\Rset^N} \abs{\nabla v}^2 + W \abs{v}^2\st v \in H^1_{\mathrm{loc}} (\Rset^N),\\
 &\qquad \qquad \int_{\Rset^N} \bigl(I_\alpha \ast \abs{v}^p\bigr) \abs{v}^p = \int_{\Rset^N} \abs{\nabla v}^2 + W \abs{v}^2 ,\\
 &\qquad\qquad W \in C (\Rset^N) \text{ is positive and \(\gamma\)--homogeneous }\\
 &\qquad \qquad \text{and there exists } x \in \Rset^N \text{ such that }\\
 &\hspace{4cm} \lim_{z \to x } \frac{V  (z) - W (z - x )}{\abs{z - x}^\gamma} = 0
 \Bigr\}<+\infty.
\end{split}
\]
\end{theorem}

Here, a function \(W : \Rset^N \to \Rset\) is \emph{positive \(\gamma\)--homogeneous} if
for every \(y \in \Rset^N \setminus \{0\}\), \(W (y) > 0\) and if for every \(t \in [0, +\infty)\) and every \(y \in \Rset^N\), \(W (t y) = t^\gamma W (y)\).

For the nonlinear Schr\"odinger equation, the semiclassical limit has been studied under similar asymptotic homogeneity conditions on the external potential \(V\) \cite{ByeonWang2002}.

The results can be restated in terms of convergence to minimizers of a \emph{concentration function}.
Indeed, if the limiting functional \(\mathcal{J}_W \in C^1 (H^1_W (\Rset^N))\) is defined for every \(v \in H^1_W (\Rset^N)\) by
\begin{equation}
\label{defJWv}
  \mathcal{J}_W (v) = \frac{1}{2} \int_{\Rset^N} \abs{\nabla v}^2 + W \abs{v}^2
  - \frac{1}{2 p} \int_{\Rset^N} \bigl(I_\alpha \ast \abs{v}^p\bigr) \abs{v}^p
\end{equation}
on the Hilbert space \(H^1_W (\Rset^N)\) obtained by completion of the set of smooth functions \(C^\infty_c (\Rset^N)\) endowed with the norm associated to the quadratic part of \(\mathcal{J}_W\):
\[
 \norm{u}_{H^1_W} = \Bigl(\int_{\Rset^N} \abs{\nabla u}^2 + W \abs{u}^2 \Bigr)^\frac{1}{2};
\]
the limiting groundstate level \(\mathcal{E} (W)\) is defined by
\begin{equation}
\label{defEW}
 \mathcal{E} (W) = \inf \,\bigl\{\mathcal{J}_W (v) \st v \in H^1_W (\Rset^N) \setminus \{0\} \text{ and }\langle \mathcal{J}_W' (v), v\rangle = 0 \bigr\}
\end{equation}
(this infimum is in fact always achieved since the positive \(\gamma\)--homogeneous potential \(W\) is coercive \cite{VanSchaftingenXia}),
then the function \(v_*\) achieves the infimum in \eqref{defEW} and
\[
 \mathcal{C} (x_*) = \inf\,\{ \mathcal{C} (x) \st x \in \Rset^N\} < + \infty,
\]
where the concentration function \(\mathcal{C} : \Rset^N \to (0, +\infty]\) is defined for each \(x \in \Rset^N\) by
\begin{equation*}
  \mathcal{C} (x)
  =
  \left\{
  \begin{aligned}
     &\mathcal{E} (W) &&\text{if \(W \in C (\Rset^N)\) is positive and \(\gamma\)--homogeneous }\\
     & && \qquad\qquad \qquad  \text{and }\lim_{z \to x} \frac{V (z)  - W (z - x)}{\abs{z - x}^\gamma} = 0,\\
     &+ \infty &&\text{if } \lim_{z \to x} \frac{V (z)}{\abs{z - x}^\gamma} = +\infty.
  \end{aligned}
  \right.
\end{equation*}

The main difficulty in order to prove Theorem~\ref{theoremGammaHomogeneous} is in the proof of the lower bound, where we have two radically different behaviour at points and these limit cannot be uniform. Our approach to this problem is to consider at every point all the homogeneous potentials that are asymptotically below the potential \(V\); at most points this class is unbounded, giving an infinity lower bound, and at the other points it is bounded and gives the lower bound.

\bigbreak

The case where the potential \(V\) vanishes on a large set has also been studied 
for the nonlinear Schr\"odinger equation \cite{ByeonWang2002}, we consider such a case for the Choquard equation.

\begin{theorem}\label{theoremDeadCore}
Let $N\in \Nset_*$, $p \in (0, +\infty)$ and  \(V \in C (\Rset^N)\).
If
\[
 \frac{N-2}{N+\alpha} < \frac{1}{p} < \frac{N}{N+\alpha},
\]
if \(V \ge 0\) on \(\Rset^N\), if
\[
 \varliminf_{\abs{x} \to \infty} V (x) > 0,
\]
and if there exists a bounded open set \(\Omega\) with a smooth boundary such that
\[
 \bigl\{ x \in \Rset^N \st V (x) = 0\bigr\} = \Bar{\Omega}\neq\emptyset,
\]
then, for sufficiently small $\varepsilon>0$, the Choquard equation~\eqref{eqChoquard} has a positive groundstate solution $u_\varepsilon \in H^1 (\Rset^N)$. \\
Moreover, there exist \(v_* \in H^1 (\Rset^N)\) which is a ground state of
\[
\left\{
\begin{aligned}
 -\Delta v_* &= \bigl(I_\alpha \ast \abs{v_*}^p \bigr)\abs{v_*}^{p - 2} v_*
 & &\text{in \(\Omega\)},\\
 v_* &= 0 & & \text{on \(\Rset^N \setminus \Omega\)},
\end{aligned}
\right.
\]
and a sequence \((\varepsilon_n)_{n \in \Nset}\) in \((0, +\infty)\) converging to \(0\) such that
\[
  \varepsilon_n^{-\frac{1}{p-1}} u_{\varepsilon_n} \to  v_* \text{ in } H^1 (\mathbb{R}^N).
\]
\end{theorem}

Theorem~\ref{theoremDeadCore} is reminiscent of some results obtained for the problem
\[
 -\Delta u_\mu + (1 + \mu V)u_{\mu} = \bigl(I_\alpha \ast \abs{u_\mu}^p) \abs{u_\mu}^{p - 2} u_\mu
\]
when \(\mu \to +\infty\) \citelist{\cite{AlvesNobregaYang2016}\cite{Lu2015}}.
\bigbreak

The rest of the paper is organized as follows. We study the existence of solutions for small parameters in Section~\ref{sectionPreliminary} and Section~\ref{sectionGroundstates} is devoted to study the existence of groundstate solution for small parameter thus complete the proof of the first part of Theorems~\ref{theoremGammaHomogeneous} and \ref{theoremDeadCore}. The asymptotics of Theorem~\ref{theoremGammaHomogeneous} are obtained in Section~\ref{sectionGammaHomogeneous}, whereas those of Theorem~\ref{theoremDeadCore} are the object of Section~\ref{sectionDeadCore}.

\section{Existence of solutions}
\label{sectionPreliminary}
Equation \eqref{eqChoquard} is variational in nature, its weak solutions are, at least formally, critical points of the functional
defined by
\begin{equation}\label{Enerfuntl}
\mathcal{I}_\varepsilon(u):=\frac{1}{2}\int_{\mathbb{R}^N}\varepsilon^2|\nabla u|^2+V \abs{u}^2-\frac{1}{2p}\int_{\mathbb{R}^N} (I_\alpha \ast \abs{u}^p)\abs{u}^p.
\end{equation}
The linear part of the equation \eqref{eqChoquard} naturally induces a norm
\[
  \norm{u}^{2}_{\varepsilon}:=\int_{\mathbb{R}^N}\varepsilon^2|\nabla u|^2+V \abs{u}^2.
\]
The norms for various \(\varepsilon\) are all equivalent to each other.
We set $H^1_V (\Rset^N)$ to be the Hilbert space obtained by
completion of the set of smooth test functions $C^{\infty}_c(\mathbb{R}^N)$ with respect to any of the norm $\norm{\cdot}_\varepsilon$.
Although it will not play any role in this work, using the continuity of \(V\) and the fact that \(\varliminf_{\abs{x} \to \infty} V (x) > 0\) the space \(H^1_V (\Rset^N)\) can also be characterized as
\[
  H^1_V (\Rset^N) =\Big\{u\in H^{1}_{\mathrm{loc}}(\mathbb{R}^N) \st \int_{\mathbb{R}^N} V\abs{u}^2<+\infty\Big\}.
\]

We first recall how the space $H^1_V (\Rset^N)$ can be embedded continuously into the classical Sobolev space $H^1(\mathbb{R}^N)$ equipped with the standard norm $\norm{\cdot}_{H^1}$ for fixed $\varepsilon>0$, even though the potential $V$ has a nontrivial set of zeroes.

\begin{lemma}\label{lem2.1}
Let \(V : \Rset^N \to \Rset\). If \(\varliminf_{\abs{x} \to \infty} V (x)  > 0\), then for every \(\varepsilon > 0\), there exists a constant \(C > 0\) such that for every \(u \in H^1_V (\Rset^N)\), \(u \in H^1 (\Rset^N)\) and
\[
 \int_{\Rset^N} \varepsilon^2 \abs{\nabla u}^2 + \abs{u}^2
 \le C \int_{\Rset^N} \varepsilon^2 \abs{\nabla u}^2 + V \abs{u}^2.
\]
\end{lemma}
\begin{proof}
Let
\[
 \mu = \frac{1}{2} \varliminf_{\abs{x} \to \infty} V (x)  > 0.
\]
By definition of the limit, there exists \(R > 0\) such that if \( x \in \Rset^N \setminus B_{R/2}\),
\[
 V(x) \ge \mu.
\]
(Here and in the sequel, we use the notation  $B_r(a)$ to denote the ball in $\mathbb{R}^N$ with radius $r$ and centered at $a$ and $B_r = B_r(0)$.)
By integration, we have immediately
\begin{equation}
\label{ineqMuOutside}
\int_{\mathbb{R}^N \setminus B_{R/2}}\abs{u}^2  \leq \frac{1}{\mu} \int_{\mathbb{R}^N \setminus B_{R/2} }V \abs{u}^2.
\end{equation}
We take  a function $\psi\in C^{\infty}(\mathbb{R}^N)$ such that $0\leq \psi\leq1$ in \(\mathbb{R}^N\), $\psi(x)=1$ for each $x\in B_{R/2}$ and $\psi (x)= 0$ for each $x \in \mathbb{R}^N \setminus B_R$. Then, it follows from the classical Poincar\'e inequality on the ball \(B_R\) that
\begin{equation}
\label{ineqMuInside}
\begin{split}
 \int_{B_{R/2}}  \abs{u}^2 &\le  \int_{B_R} \abs{\psi u}^2 \leq \Cl{upoinc} \int_{B_R} \abs{\nabla(\psi u)}^2\\
&\leq 2 \Cr{upoinc} \int_{\mathbb{R}^N} \abs{\nabla u}^2+ \frac{2 \Cr{upoinc} \norm{ \nabla \psi}_{L^\infty}^2}{\mu} \int_{B_R \setminus B_{R/2}} V \abs{u}^2 .
\end{split}
\end{equation}
The conclusion then follows from the combination of the inequalities \eqref{ineqMuInside} and \eqref{ineqMuOutside}.
\end{proof}

By the classical Sobolev embedding of $H^{1}(\mathbb{R}^N)$ to $L^q(\mathbb{R}^N)$,
we deduce that the space \(H^1_V (\Rset^N)\) can be continuously embedded in \(L^q (\Rset^N)\) when $\frac{1}{2}\ge \frac{1}{q}\geq\frac{1}{2}-\frac{1}{N}$.

The well-definiteness, continuity and differentiability of the nonlocal term in the function \(\mathcal{I}_\varepsilon\) defined by \eqref{Enerfuntl} follows then from the classical \emph{Hardy--Littlewood--Sobolev inequality} \cite{LL}*{Theorem 4.3} which states that if $\alpha\in(0,N)$, $s\in(1,N/\alpha)$ and if $\varphi\in L^s(\mathbb{R}^N)$, then $I_\alpha \ast \varphi\in L^{\frac{Ns}{N-\alpha s}}(\mathbb{R}^N)$ and
\begin{equation}
\label{HLS}
 \norm{I_\alpha \ast \varphi}_{L^{\frac{Ns}{N-\alpha s}}}\leq C_{H}\|\varphi\|_{L^s}
\end{equation}
where the constant $C_{H}>0$ depends only on $\alpha$, $N$, and $s$.

A solution $u$ is a \emph{groundstate} of the Choquard equation \eqref{eqChoquard}  $\mathcal{I}_\varepsilon(u)$ is the least among all nontrivial critical values of $\mathcal{I}_\varepsilon$,
namely, $u$ has the least energy among nontrivial solutions.
A natural and well known method to search the groundstate is to minimize the functional $\mathcal{I}_\varepsilon$ on the \emph{Nehari manifold} (see \cite{SW2}) of the equation \eqref{eqChoquard} which is defined by
\[
\mathcal{N}_\varepsilon:=\bigl\{u\in H^1_V(\Rset^N) \st u\neq 0 \text{ and } \langle \mathcal{I}_{\varepsilon}'(u),u\rangle=0\bigr\}.
\]
The corresponding groundstate energy is described as
\begin{equation*}
c_\varepsilon:=\inf\limits_{u\in\mathcal{N}_\varepsilon} \mathcal{I}_\varepsilon(u).
\end{equation*}
\begin{lemma}
\label{lem2.3}
Let $p>1$ and $1/p\in(\frac{N-2}{N+\alpha},\frac{N}{N+\alpha})$. For given $\varepsilon>0$, the groundstate energy $c_\varepsilon$  is positive and $\mathcal{N}_\varepsilon$ is a manifold of class $C^1$. Moreover, if $u\in\mathcal{N}_\varepsilon$ is a critical point of $ \mathcal{I}_\varepsilon\big|_{\mathcal{N}_\varepsilon}$, then $ \mathcal{I}_{\varepsilon}'(u)=0$.
\end{lemma}
\begin{proof}%
\resetconstant
We fix \(\varepsilon > 0\).
If we define $\mathcal{G}_\varepsilon(u):=\langle  \mathcal{I}_{\varepsilon}'(u),u\rangle$,
then for any $u\in \mathcal{N}_\varepsilon$, we have $\mathcal{G}_\varepsilon(u)=0$, which, together with the Hardy--Littlewood--Sobolev inequality \eqref{HLS} and the Sobolev inequality implies that
\[
  \norm{u}_{\varepsilon}^2=\int_{\mathbb{R}^N} (I_\alpha \ast \abs{u}^p)\abs{u}^p\leq \Cl{xays} \norm{u}_{\varepsilon}^{2p}
\]
where the constant $\Cr{xays} >0$ depends on $\varepsilon$.
This leads to
\begin{equation*}
\norm{u}^2_\varepsilon\geq\Cr{xays}^{-\frac{1}{p-1}}>0.
\end{equation*}
Hence, for any $u\in\mathcal{N}_\varepsilon$, we have
\[
 \mathcal{I}_\varepsilon(u)
 =\Bigl(\frac{1}{2}-\frac{1}{2p}\Bigr)\norm{u}_{\varepsilon}^{2}
 \geq \Bigl(\frac{1}{2}-\frac{1}{2p}\Bigr)C_1^{-\frac{1}{p-1}}
\]
thus, $c_\varepsilon>0$. Furthermore, since $p>1$, we know, for each $u\in\mathcal{N}_\varepsilon$, that
\begin{equation}\label{boundaway2}
\langle \mathcal{G}_\varepsilon '(u),u\rangle =-2(p-1)\norm{u}_{\varepsilon}^2\leq -2(p-1)C_1^{-\frac{1}{p-1}}<0,
\end{equation}
it follows from the implicit function theorem that $\mathcal{N}_\varepsilon$ is an embedded submanifold of class $C^1$.

Let us no assume that the function $u\in\mathcal{N}_\varepsilon$ is a critical point of  the restricted functional $\mathcal{I}_\varepsilon\big|_{\mathcal{N}_\varepsilon}$, then there exists a Lagrange multiplier $\lambda_\varepsilon\in\mathbb{R}$, such that 
\begin{equation}
\label{eqxanlg}
\mathcal{I}_{\varepsilon}'(u)=\lambda_\varepsilon G_{\varepsilon}'(u). 
\end{equation}
By testing this equation against \(u\) itself, we have
\[
  \lambda_\varepsilon \langle G_{\varepsilon}'(u),u\rangle=\langle \mathcal{I}_{\varepsilon}'(u),u\rangle =0.
 \]
 we thus deduce by \eqref{boundaway2} that $\lambda_\varepsilon=0$ and the conclusion follows then from \eqref{eqxanlg}.
\end{proof}

We now prove the existence of groundstate solutions of \eqref{eqChoquard} for small parameters.

\begin{proposition}
\label{propositionExistence}
Let $N\geq 1$, $p>1$, $1/p\in(\frac{N-2}{N+\alpha},\frac{N}{N+\alpha})$ and let
\(V \in C (\Rset^N)\).
If
$$
 0 \le \inf_{\Rset^N} V < \varliminf\limits_{\abs{x}\to\infty}V(x),
$$
then, for sufficiently small $\varepsilon>0$, the Choquard equation~\eqref{eqChoquard} has a positive groundstate solution.
\end{proposition}

Proposition~\ref{propositionExistence} is a counterpart for the Choquard equation of Rabinotwitz's existence result for the nonlinear Schr\"odinger equation \cite{Rabinowitz1992}*{Theorem 4.33}.
\begin{proof}[Proof of Proposition~\ref{propositionExistence}]
By Ekeland's variational principle \cite{W}, there exists a minimizing sequence $(u_n)_{n\in\mathbb{N}}$ in $\mathcal{N}_\varepsilon$ for $c_\varepsilon$, such that, as \(n \to \infty\),
\begin{gather*}
\mathcal{I}_\varepsilon (u_n) \to c_\varepsilon ,\,\, \text{ and }\,\,
\mathcal{I}_\varepsilon'(u_n)-\lambda_n \mathcal{G}_{\varepsilon}'(u_n) \to 0 \qquad \text{ in } \bigl(H^1_V (\Rset^N)\bigr)'.
\end{gather*}
We first observe that the sequence \((u_n)_{n \in \Nset}\) is bounded, because
\[
\frac{p-1}{2p}\norm{u_n}^{2}_{\varepsilon}=\mathcal{I}_\varepsilon(u_n)-\frac{1}{2p}\langle \mathcal{I}_{\varepsilon}'(u_n),u_n\rangle\leq c_\varepsilon+o_n(\norm{u_n}_\varepsilon) .
\]
It follows then that \(\lambda_n \langle \mathcal{G}_\varepsilon'(u_n),u_n\rangle \to 0\) as \(n \to \infty\). By \eqref{boundaway2} again, we see that $\lambda_n\to 0$  as \(n \to \infty\). Note that $ \mathcal{G}_\varepsilon'(u_n)$ is bounded in the dual space  $\big(H^1_V (\Rset^N)\big)'$, in fact, for every \(\varphi \in H^1_V (\Rset^N)\), we have
\begin{equation*}
\begin{split}
|\langle \mathcal{G}_\varepsilon'(u_n),\varphi\rangle|&\leq 2\langle u_n,\varphi \rangle+2p\int_{\Rset^N}\bigl(I_\alpha*|u_n|^p\bigr)|u_n|^{p-2}u_n\varphi\\
&\leq C\bigl(\|u_n\|_\varepsilon+C\|u_n\|_{\varepsilon}^{2p-1}\bigr)\|\varphi\|_\varepsilon\leq C\|\varphi\|_\varepsilon.
\end{split}
\end{equation*}
Hence, $\mathcal{I}_{\varepsilon}'(u_n)\to 0$ as $n\to\infty$  in  $\big(H^1_V (\Rset^N)\big)'$.

Up to a subsequence we can assume that $u_n\rightharpoonup u$ weakly in $H^1_V (\Rset^N)$ and $u_n\to u$ almost everywhere in $\mathbb{R}^N$ as \(n \to \infty\). If \(u \ne 0\), we reach the conclusion. Indeed, $\mathcal{I}_\varepsilon'(u)=0$, which, together with the weakly lower semi-continuity of the norm, implies that
\begin{equation*}
\begin{split}
c_\varepsilon +o_n(1)\norm{u_n}_\varepsilon&=\mathcal{I}_\varepsilon(u_n)-\frac{1}{2p}\langle \mathcal{I}_\varepsilon'(u_n),u_n\rangle = \Bigl(\frac{1}{2}-\frac{1}{2p}\Bigr)\norm{u_n}_{\varepsilon}^{2}\\
&\geq \Bigl(\frac{1}{2}-\frac{1}{2p}\Bigr)\norm{u}_{\varepsilon}^{2} =\mathcal{I}_\varepsilon(u)-\frac{1}{2p}\langle \mathcal{I}_\varepsilon'(u),u\rangle=\mathcal{I}_\varepsilon(u)\geq c_\varepsilon\\
\end{split}
\end{equation*}
that is, the function $u$ is a minimizer for $c_\varepsilon$ and is thus a groundstate of the Choquard equation \eqref{eqChoquard} by Lemma~\ref{lem2.3}.

In order to conclude, we assume by contradiction that \(u = 0\).
We have then \(u_n \to 0\) in \(L_{\rm{loc}}^{{2 N p}/(N + \alpha)} (\Rset^N)\) as \(n \to \infty\).
We choose \(R > 0\), \(\delta>0\), \(\mu > \nu > 0\) and \(x_* \in \Rset^N\), such that
\begin{align*}
 V (x) &\ge \mu & &\text{ if } \abs{x} \ge R , &&\text{ and } &V (x)&\leq \nu &&\text{ for every } x\in B_{\delta}(x_*).
\end{align*}
and a cut-off function \(\eta \in C^\infty (\Rset^N)\) such that \(0 \le \eta \le 1\) in \(\Rset^N\), \(\eta = 0\) in \(B_R\) and \(\eta = 1\) on \(\Rset^N \setminus B_{2 R}\). We define the function \(v_n = \eta u_n\).
We have by our contradiction assumption, as \(n \to \infty\),
\begin{equation*}
\begin{split}
 \int_{\Rset^N} \varepsilon^2 \abs{\nabla v_n}^2 + \mu \abs{v_n}^2 &\le \int_{\Rset^N} \varepsilon^2 \abs{\nabla v_n}^2 + V \abs{v_n}^2 \\
 &\le \int_{\Rset^N} \varepsilon^2 \abs{\nabla u_n}^2 +  V \abs{u_n}^2 + o_n (1),
\end{split}
\end{equation*}
and
\[
 \int_{\Rset^N} (I_\alpha \ast \abs{v_n}^p) \abs{v_n}^p = \int_{\Rset^N} \bigl(I_\alpha \ast \abs{u_n}^p\bigr) \abs{u_n}^p + o_n (1).
\]
If we define
\[
 \mathcal{I}^\mu_\varepsilon (v) = \frac{1}{2} \int_{\Rset^N} \varepsilon^2 \abs{\nabla v}^2 + \mu \abs{v}^2
 - \frac{1}{2 p} \int_{\Rset^N} \bigl(I_\alpha \ast \abs{v}^p\bigr) \abs{v}^p,
\]
and if we take \(t_n \in (0, +\infty)\) such that
\[
 t_n^2 \int_{\Rset^N} \varepsilon^2 \abs{\nabla v_n}^2 + \mu \abs{v_n}^2
 = t_n^{2 p} \int_{\Rset^N} (I_\alpha \ast \abs{v_n}^p) \abs{v_n}^p,
\]
then $\limsup_{n \to \infty} t_n\leq 1$ and thus
\begin{equation}
\label{ineqLowerContradiction}
 c_\varepsilon = \lim_{n \to \infty} \mathcal{I}_\varepsilon (u_n)
 \ge \varliminf_{n \to \infty} \mathcal{I}_\varepsilon (t_n u_n) \ge \varliminf_{n \to \infty} \mathcal{I}^\mu_\varepsilon (t_n v_n)\ge c_{\varepsilon}^\mu,
\end{equation}
where
\[
 c_{\varepsilon}^\mu = \inf\, \bigl\{ \mathcal{I}^\mu_\varepsilon (v) \st v \in H^1 (\Rset^N)
\setminus\{0\} \text{ and }
  \langle {\mathcal{I}_\varepsilon^\mu}' (v), v \rangle = 0\bigr\}.
\]

Since \(\nu < \mu\), we have \(c_1^\nu < c_1^\mu\) and there exists a function \(\varphi \in C^\infty_c (\Rset^N) \setminus \{0\}\),
such that \(\langle \mathcal{I}_1^\nu{}' (\varphi), \varphi \rangle = 0\) and
\[
  \mathcal{I}_1^\nu{} (\varphi) < c_1^\mu.
\]
We next let
\(\varphi_\varepsilon (x) = \varepsilon^\frac{-\alpha}{2 (p - 1)} \varphi (\varepsilon^{-1}(x - x_*))\), and we observe that \(\langle \mathcal{I}_\varepsilon^\nu{}' (\varphi_\varepsilon), \varphi_\varepsilon \rangle = 0\) and
\[
 \mathcal{I}^{\nu}_\varepsilon (\varphi_\varepsilon) <c_{\varepsilon}^\mu.
\]
On the other hand, by the definition of $c_\varepsilon$, we have
\begin{equation*}
\begin{split}
c_\varepsilon &\leq \max_{t\geq 0}\mathcal{I}_{\varepsilon}(t\varphi_\varepsilon)
=\Big(\frac{1}{2}-\frac{1}{2p}\Big)
\frac{\displaystyle \Bigl(\int_{\Rset^N}\varepsilon^2\abs{\nabla \varphi_\varepsilon}^2+V \abs{\varphi_\varepsilon}^2\Bigr)^\frac{p}{p - 1}}{\Bigl(\displaystyle \int_{\Rset^N}(I_\alpha*\abs{\varphi_\varepsilon}^p)
\abs{\varphi_\varepsilon}^p\Bigr)^{\frac{1}{p - 1}}}\\
%
\end{split}
\end{equation*}
When $\varepsilon$ is small enough so that \(\varepsilon \supp \varphi \subset B_\delta (x_*)\), we have \(V \le \nu\) on \(x_* + \varepsilon \supp \varphi\) and we conclude that
\[
c_\varepsilon 
\leq \Big(\frac{1}{2}-\frac{1}{2p}\Big)
\frac{\displaystyle \Bigl(\int_{\Rset^N}\varepsilon^2\abs{\nabla \varphi_\varepsilon}^2+\nu \abs{\varphi_\varepsilon}^2\Bigr)^\frac{p}{p - 1}}{\Bigl(\displaystyle \int_{\Rset^N}(I_\alpha*\abs{\varphi_\varepsilon}^p)
\abs{\varphi_\varepsilon}^p\Bigr)^{\frac{1}{p - 1}}} \le \mathcal{I}_{\varepsilon}^\nu(\varphi_\varepsilon)<c_{\varepsilon}^\mu,
\]
which contradicts the lower bound \eqref{ineqLowerContradiction}.
\end{proof}

\section{Asymptotics for potential with homogeneous zeroes}
\label{sectionGammaHomogeneous}
\subsection{Asymptotic upper bound}
\label{sectionGroundstates}

We define the \emph{upper concentration function} \(\Bar{\mathcal{C}} : \Rset^N \to \Rset\) for \(x \in \Rset^N\) by
\begin{multline*}
  \Bar{\mathcal{C}} (x)
  = \inf\,\Bigl\{\mathcal{E} (\Bar{W}) \st \Bar{W} \in C (\Rset^N)  \text{ is positive and \(\gamma\)--homogeneous and } \\
  \varlimsup_{z \to x} \frac{V (z)  - \Bar{W} (z - x)}{\abs{z - x}^\gamma} \le 0 \Bigr\}.
\end{multline*}
The quantity \(\mathcal{E} (\Bar{W})\) was defined
in \eqref{defEW} as the groundstate energy of the limiting functional \(\mathcal{J}_{\Bar{W}}\) defined in \eqref{defJWv} on the weighted Sobolev space \(H^1_{\Bar{W}} (\Rset^N)\).

The assumptions of Theorem~\ref{theoremGammaHomogeneous} ensure that \(\Bar{\mathcal{C}} = \mathcal{C}\)
everywhere in \(\Rset^N\). Indeed, if \(\lim_{z \to x} V (x) /\abs{z- x}^\gamma = +\infty\), then there is no function \(\Bar{W}\) satisfying the condition and thus \(\Bar{\mathcal{C}} (x) = +\infty = \mathcal{C} (x)\). Otherwise, there exists a positive \(\gamma\)--homogeneous function  \(W \in C (\Rset^N)\) such that
\[
 \lim_{z \to x} \frac{V (z)  - W (z - x)}{\abs{z - x}^\gamma} = 0,
\]
and thus we have \(\mathcal{C} (x) = \mathcal{E} (W) \ge \Bar{\mathcal{C}} (x)\).
Moreover, if \(\Bar{W} \in C (\Rset^N)\) is positive and \(\gamma\)--homogeneous and if
\[
  \lim_{z \to x} \frac{V (z)  - \Bar{W} (z - x)}{\abs{z - x}^\gamma} \le 0,
\]
then \(W \le \Bar{W}\) in \(\Rset^N\), and thus \(\mathcal{E} (W) \le \mathcal{E} (\Bar{W})\), so that by taking the infimum, \(\mathcal{C} (x) \le \Bar{\mathcal{C}} (x)\).

To alleviate the notation, we fix for the rest of this section
\[
  \kappa=\frac{2}{\gamma+2}.
\]

\begin{proposition}\label{propositionUpperBound}
One has
\[
\varlimsup_{\varepsilon\to 0}\frac{c_\varepsilon}{\varepsilon^{\kappa (N + \frac{p \gamma - \alpha}{p - 1})}} \leq \inf_{\Rset^N} \Bar{\mathcal{C}}.
\]
\end{proposition}
\begin{proof}
Let \(x_* \in \Rset^N\), let \(\Bar{W} \in C (\Rset^N)\) be a positive \(\gamma\)--homogeneous function
such that
\begin{equation}
\label{eqUpperBoundApproxW}
 \lim_{x \to x_*} \frac{V (x) - \Bar{W} (x - x_*)}{\abs{x - x_*}^\gamma} \le 0
\end{equation}
and let \(\varphi \in C^\infty_c (\Rset^N) \setminus \{0\}\).
For \(\varepsilon > 0\), we define the function \(\varphi_\varepsilon : \Rset^N \to \Rset\) for each \(x \in \Rset^N\) by
\[
 \varphi_\varepsilon (x) = \varepsilon^{{\kappa}\frac{\gamma - \alpha}{2(p - 1)}} \varphi \Bigl(\frac{x - x_*}{\varepsilon^\kappa} \Bigr).
\]
We observe that by homogeneity and scaling, we have for each \(\varepsilon > 0\), 
\begin{gather*}
 \int_{\Rset^N} \varepsilon^2 \abs{\nabla \varphi_\varepsilon}^2
 = \varepsilon^{\kappa (N + \frac{p \gamma - \alpha}{p - 1})}\int_{\Rset^N} \abs{\nabla \varphi}^2,\\
 \int_{\Rset^N} V \abs{\varphi_\varepsilon}^2
 = \varepsilon^{\kappa (N + \frac{p \gamma - \alpha}{p - 1})} \int_{\Rset^N} \frac{V (x_* + \varepsilon^\kappa y)}{\varepsilon^{\kappa \gamma}} \abs{\varphi (y)}^2 \dif y, \\
 \int_{\Rset^N} \bigl(I_\alpha \ast \abs{\varphi_\varepsilon}^p\bigr) \abs{\varphi_\varepsilon}^p
 = \varepsilon^{\kappa (N + \frac{p \gamma - \alpha}{p - 1})} \int_{\Rset^N} (I_\alpha \ast \abs{\varphi}^p) \abs{\varphi}^p.
\end{gather*}
Since the function \(\Bar{W}\) is \(\gamma\)--homogeneous and satisfies \eqref{eqUpperBoundApproxW}, we have for each \(y \in \Rset^N\),
\[
 \lim_{\varepsilon \to 0}  \frac{V (x_* + \varepsilon^\kappa y)}{\varepsilon^{\kappa \gamma}}
 = \Bar{W} (y) + \abs{y}^\gamma \lim_{\varepsilon \to 0} \frac{V (x_* + \varepsilon^\kappa y) - \Bar{W} (\varepsilon^\kappa y)}{\abs{\varepsilon^{\kappa} y}^\gamma} = \Bar{W} (y),
\]
uniformly when \(y\) stays in the support of \(\varphi\) which is compact by assumption.
Thus by Lebesgue's dominated convergence theorem, we have
\[
 \lim_{\varepsilon \to 0} \frac{1}{\varepsilon^{\kappa (N + \frac{p \gamma - \alpha}{p - 1})}} \int_{\Rset^N} V \abs{\varphi_\varepsilon}^2
 = \int_{\Rset^N} \Bar{W} \abs{\varphi}^2.
\]
For every \(\varepsilon > 0\), we fix \(t_\varepsilon \in (0, +\infty)\) in such a way that
\[
 t_\varepsilon^2 \int_{\Rset^N} \varepsilon^2 \abs{\nabla \varphi_\varepsilon}^2 + V \abs{\varphi_\varepsilon}^2 =
 t_\varepsilon^{2 p} \int_{\Rset^N} \bigl(I_\alpha \ast \abs{\varphi_\varepsilon}^p\bigr) \abs{\varphi_\varepsilon}^p,
\]
and we observe that \(\lim_{\varepsilon \to 0} t_\varepsilon = t_*\), where \(t_* \in (0, +\infty)\) is characterized by 
\[
 t_*^2 \int_{\Rset^N} \abs{\nabla \varphi}^2 + \Bar{W} \abs{\varphi}^2 =
 t_*^{2 p} \int_{\Rset^N} \bigl(I_\alpha \ast \abs{\varphi}^p\bigr) \abs{\varphi}^p.
\]
We have then
\[
\begin{split}
 \varlimsup_{\varepsilon \to 0} \frac{c_\varepsilon}{\varepsilon^{\kappa (N + \frac{p \gamma - \alpha}{p - 1})}}
 \le \lim_{\varepsilon \to 0} \sup_{t \in (0, +\infty)} \frac{\mathcal{I}_{\varepsilon} (t \varphi_\varepsilon)}{\varepsilon^{\kappa (N + \frac{p \gamma - \alpha}{p - 1})}}
 &= \lim_{\varepsilon \to 0} \frac{\mathcal{I}_{\varepsilon} (t_\varepsilon \varphi_\varepsilon)}{\varepsilon^{\kappa (N + \frac{p \gamma - \alpha}{p - 1})}}\\
 & = \mathcal{J}_{\Bar{W}} (t_* \varphi) = \sup_{t \in (0, +\infty)} \mathcal{J}_{\Bar{W}} (\varphi).
\end{split}
\]
Since the left-hand side is independent of \(\varphi\), taking the infimum with respect to \(\varphi\) and by density of the set of smooth test functions \(C^\infty_c (\Rset^N)\) in
the weighted space \(H^1_{\Bar{W}} (\Rset^N)\), we have
\[
 \varlimsup_{\varepsilon \to 0} \frac{c_\varepsilon}{\varepsilon^{\kappa (N + \frac{p \gamma - \alpha}{p - 1})}}
 \le \mathcal{E} (\Bar{W}).
\]
Since the left-hand side does not depend on \(\Bar{W}\), by taking now the infimum with respect to suitable positive \(\gamma\)--homogeneous functions \(\Bar{W} \in C (\Rset^N)\), we deduce that
\[
 \varlimsup_{\varepsilon \to 0} \frac{c_\varepsilon}{\varepsilon^{\kappa (N + \frac{p \gamma - \alpha}{p - 1})}}
 \le \Bar{\mathcal{C}} (x_*).
\]
Since the point \(x_* \in \Rset^N\) is arbitrary, the conclusion follows.
\end{proof}

\subsection{Asymptotic lower bound and behaviour of solutions}
\label{sectionConcentration}
We define the lower concentration function \(\underline{\mathcal{C}} : \Rset^N \to \Rset\) by
\begin{multline*}
  \underline{\mathcal{C}} (x)
  = \sup \Bigl\{\mathcal{E} (W) \st W \in C (\Rset^N)  \text{ is positive and \(\gamma\)--homogeneous and } \\
  \varliminf_{z \to x} \frac{V (z)  - W (z - x)}{\abs{z - x}^\gamma} \ge 0 \Bigr\},
\end{multline*}
where the quantity \(\mathcal{E} (W)\) was defined in \eqref{defEW}.

Under the assumptions of Theorem~\ref{theoremGammaHomogeneous}, we have \(\underline{\mathcal{C}} = \mathcal{C}\). Indeed, if \(\lim_{z \to x} V (x) /\abs{z- x}^\gamma = +\infty\), then we can take any positive and \(\gamma\)--homogeneous function in the definition of the lower concentration function \(\underline{\mathcal{C}}\) and thus \(\underline{\mathcal{C}} (x) = +\infty \).  Otherwise, there exists a positive  and \(\gamma\)--homogeneous function  \(W \in C (\Rset^N)\) such that
\[
 \lim_{z \to x} \frac{V (z)  - W (z - x)}{\abs{z - x}^\gamma} = 0
\]
and thus \(\mathcal{C} (x) \le \underline{\mathcal{C}} (x)\).
Moreover, if
\[
 \lim_{z \to x} \frac{V (z)  - \underline{W} (z - x)}{\abs{z - x}^\gamma} \ge 0,
\]
then \(\underline{W} \le W\) on \(\Rset^N\) and by monotonicity of \(\mathcal{E}\), we have \(\mathcal{E} (\underline{W})  \le \mathcal{E} (W) = \mathcal{C} (x)\); it follows then that \(\underline{\mathcal{C}} (x) \le \mathcal{C} (x)\).

\begin{proposition}
\label{propositionLowerBound}
Let \((\varepsilon_n)_{n \in \Nset}\) be a sequence of positive numbers converging to \(0\) and \((u_n)_{n \in \Nset}\) be solutions in \(H^1_V (\Rset^N)\) of problem (\(\mathcal{C}_{\varepsilon_n}\)).
If
\[
 \varliminf_{n \to \infty} \frac{1}{\varepsilon_n^{\kappa (N + \frac{p \gamma - \alpha}{p - 1})}} \mathcal{I}_{\varepsilon_n} (u_n) < + \infty.
\]
Then up to a subsequence, there exists \(R_* > 0\) and \(x_* \in \Rset^N\) such that
\[
  \lim_{n \to \infty} \frac{1}{\varepsilon_n^{\kappa (N + \frac{\gamma - \alpha}{p - 1})}} \int_{B_{\varepsilon_n^\kappa R_*} (x_*)} \abs{u_n}^2 > 0
\]
and
\[
\underline{\mathcal{C}}  (x_*) \le \varliminf_{n \to \infty} \frac{1}{\varepsilon_n^{\kappa (N + \frac{p \gamma - \alpha}{p - 1})}} \mathcal{I}_{\varepsilon_n} (u_n) .
\]
If moreover \(W\) is a positive \(\gamma\)--homogeneous function such that
\[
 \lim_{x \to x_*} \frac{V (x) - W (x - x_*)}{\abs{x - x_*}^\gamma} = 0,
\]
then there exists \(v_* \in H_W^1 (\Rset^N)\setminus\{0\}\) such that
\[
  \varepsilon_n^{\kappa \frac{\gamma - \alpha}{2(p - 1)}} u_n (x_* + \varepsilon_n^\kappa \cdot)
  \weakto v_* \qquad \text{weakly in \(H^1_\mathrm{loc} (\Rset^N)\)},
\]
\(v_*\) is a weak solution to
\[
 -\Delta v_* + W v_* = (I_\alpha \ast \abs{v_*}^p) \abs{v_*}^{p - 2} v_*
\]
and
\[
 \underline{\mathcal{C}} (x_*) \le \mathcal{J}_W (v_*) \le  \varliminf_{n \to \infty} \frac{1}{\varepsilon_n^{\kappa (N + \frac{p \gamma - \alpha}{p - 1})}} \mathcal{I}_{\varepsilon_n} (u_n).
\]
\end{proposition}

In order to prepare the proof of Proposition~\ref{propositionLowerBound}, we first give a lower bound on the potential \(V\).

\begin{lemma}
\label{lemmaLowerBoundVEverywhere}
Let \(V \in C (\Rset^N)\) and \(\gamma > 0\). If \(V \ge 0\) on \(\Rset^N\),
\[
 \varliminf_{\abs{x} \to \infty} V (x) > 0
\]
and if for each \(x \in \Rset^N\)
\[
 \varliminf_{z \to x} \frac{V (z)}{\abs{z - x}^\gamma} > 0,
\]
then there exist \(k \in \Nset\), \(a_1, \dotsc, a_k \in \Rset^N\), \(\mu > 0\) and \(\nu > 0\) such that for each
\(x \in \Rset^N\),
\[
  V (x) \ge \min \bigl\{\mu, \nu \abs{x - a_1}^\gamma, \dotsc, \nu \abs{x - a_k}^\gamma\bigr\}.
\]
\end{lemma}
\begin{proof}
We define the set \(K = V^{-1} (\{0\})\). Since the function \(V\) is continuous and \(\varliminf_{\abs{x} \to \infty} V (x) > 0\),
the set \(K\) is compact.
If \(x \in K\), we have \(\varliminf_{z \to x} {V (z)}/{\abs{z - x}^\gamma}>0 \) and there exists
thus \(\delta > 0\) such that \(B_\delta (x) \cap K = \{x\}\). Hence, the set \(K\) is finite and can be written as
\(K = \{a_1, \dotsc, a_k\}\) with \(k \in \Nset\) and \(a_1, \dotsc, a_k \in \Rset^N\).
Moreover, there exist  \(\rho > 0\) and \(\nu > 0\) such that if \(j \in \{1, \dotsc, k\}\) and
\(x \in B_{\rho} (a_j)\), then \(V (x) \ge \nu \abs{x - a_j}^\gamma\).
Since \(\varliminf_{\abs{x} \to \infty} V (x) > 0\), there exists \(\mu > 0\) such that
\(V (x) \ge \mu\) for every \(x \in \Rset^N \setminus \bigcup_{j= 1}^k B_\rho (a_j)\). The conclusion follows.
\end{proof}

Thanks to Lemma~\ref{lemmaLowerBoundVEverywhere}, we establish a uniform estimate on rescaled balls of $\mathbb{R}^N$, which is very useful in our subsequent arguments.

\begin{lemma}\label{lemglobal}There exists a positive number $C$, such that if \(\varepsilon\) is sufficiently small, then
for every $u\in H^1 (B_{\varepsilon^{\kappa}}(x))$ and every $x\in \mathbb{R}^N$, we have
\begin{equation*}
\int_{B_{\varepsilon^{\kappa}}(x)} \varepsilon^{2\kappa} \abs{\nabla u}^2+\abs{u}^2
\leq C \varepsilon^{-\kappa \gamma}  \int_{B_{\varepsilon^{\kappa}} (x)} \varepsilon^2 \abs{\nabla u}^2+  V \abs{u}^2.
\end{equation*}
\end{lemma}
\begin{proof}
\resetconstant
Let \(x \in \Rset^N\). By the Minkowski, Poincar\'{e} and Cauchy--Schwarz inequalities (see for example \cite{Evans}), we first see that,
\begin{equation}
\label{ineqMinkPoinc}
\begin{split}
  \bigg(\int_{B_{\varepsilon^\kappa} (x)}\abs{u}^2\bigg)^{\frac{1}{2}}
&\leq\bigg(\int_{B_{\varepsilon^\kappa}(x)}\abs{u - \Bar{u}}^2\bigg)^{\frac{1}{2}}+
\bigg(\int_{B_{\varepsilon^\kappa}(x)}|\Bar{u}|^2\bigg)^{\frac{1}{2}}\\
&\leq  \Cl{poincareitsr} \varepsilon^{\kappa} \bigg( \int_{B_{\varepsilon^\kappa}(x)}\abs{\nabla u}^2\bigg)^{\frac{1}{2}}
+\frac{1}{ |B_{\varepsilon^\kappa}|^{\frac{1}{2}}}\int_{B_{\varepsilon^\kappa}(x)}\abs{u}
\end{split}
\end{equation}
where the constant $\Cr{poincareitsr}$ only depends on the dimension $N$, and \(\Bar{u}\) denotes the average of the function
\(u\) on the ball \(B_{\varepsilon^\kappa} (x)\):
\[
\Bar{u}:=\frac{1}{|B_{\varepsilon^\kappa}(x)|} \int_{B_{\varepsilon^\kappa}(x)} u.
\]
By Lemma~\ref{lemmaLowerBoundVEverywhere}, we observe that, if \(\lambda \le \mu\),
\[
 \vert \{z \in \Rset^N \st V (z) < \lambda \}\vert \le k \abs{B_1} \,\Bigl(\frac{\lambda}{\mu}\Bigr)^\frac{N}{\gamma}.
\]
If we take \(\lambda = \mu ((1/4k)^{1/N}\varepsilon^\kappa)^{\gamma}\), we have, if \(\varepsilon\) is small enough,
\begin{equation}
\label{ineqVeps}
  \vert \{z \in \Rset^N \st V (z) < \Cl{cVeps} \varepsilon^{\kappa\gamma} \}\vert \le \frac{\abs{B_{\varepsilon^\kappa}}}{4}.
\end{equation}
We have thus, by \eqref{ineqVeps} and by the Cauchy--Schwarz inequality
\[
\begin{split}
  \int_{B_{\varepsilon^\kappa}(x)}\abs{u}
  \le \frac{\Cl{eqystr1}}{(\Cr{cVeps} \varepsilon^{\kappa\gamma})^{1/2}} \int_{B_{\varepsilon^\kappa}(x)} V^\frac{1}{2} \abs{u}
  + \int_{B_{\varepsilon^\kappa}(x) \cap V^{-1} ([0, \Cr{cVeps} \varepsilon^{\kappa\gamma}))} \abs{u}\\
  \le \frac{\Cr{eqystr1}\abs{B_{\varepsilon^\kappa}}^\frac{1}{2}}{(\Cr{cVeps} \varepsilon^{\kappa\gamma})^{1/2}}
  \biggl(\int_{B_{\varepsilon^\kappa}(x)} V \abs{u}^2\biggr)^\frac{1}{2}
  + \frac{\abs{B_{\varepsilon^\kappa}}^{1/2}}{2} \biggl(\int_{B_{\varepsilon^\kappa}(x)} \abs{u}^2\biggr)^\frac{1}{2}.
\end{split}
\]
In view of \eqref{ineqMinkPoinc} we obtain finally
\[
\begin{split}
 \bigg(\int_{B_{\varepsilon^\kappa} (x)}&\abs{u}^2\bigg)^{\frac{1}{2}}\\
 &\le  \Cr{poincareitsr}\varepsilon^\kappa \bigg( \int_{B_{\varepsilon^\kappa}(x)}\abs{\nabla u}^2\bigg)^{\frac{1}{2}}
 + \frac{\Cr{eqystr1}}{(\Cr{cVeps} \varepsilon^{\kappa\gamma})^{1/2}}  \biggl(\int_{B_{\varepsilon^\kappa}(x)} V \abs{u}^2\biggr)^\frac{1}{2}
 + \frac{1}{2} \bigg(\int_{B_{\varepsilon^\kappa} (x)}\abs{u}^2\bigg)^{\frac{1}{2}}\\
 &\le \frac{\C}{\varepsilon^{\gamma/(\gamma + 2)}} \bigg( \int_{B_{\varepsilon^\kappa}(x)}\varepsilon^2 \abs{\nabla u}^2 + V \abs{u}^2 \bigg)^{\frac{1}{2}}+ \frac{1}{2} \bigg(\int_{B_{\varepsilon^\kappa} (x)}\abs{u}^2\bigg)^{\frac{1}{2}}.
\end{split}
\]
The conclusion follows.
\end{proof}

Finally, we recall how similarly to Lemma~\ref{lem2.1}, a control in \(H^1_W(\Rset^N)\) on a ball gives a control in \(H^1\) on the same ball.
\begin{lemma}\label{lem4.3}
If \(W \in C (\Rset^N)\) is positive and \(\gamma\)--homogeneous, then there exists a constant \(C > 0\)
such that if $R > 0$ and \(v \in H^1 (B_R)\), then
\begin{equation*}
\int_{B_R} \abs{v}^2 \leq C \int_{B_R} R^2 \abs{\nabla v}^2 + \frac{W \abs{v}^2}{R^\gamma}.
\end{equation*}
\end{lemma}
\begin{proof}
\resetconstant
By scaling of the inequality and by \(\gamma\)--homogeneity of the potential \(W\), we can assume without loss of generality that \(R = 1\).
We choose \(\psi \in C^{\infty}_c (B_1)\)  such that \(\psi = 1\) on \(B_{1/2}\).
By the Poincar\'{e} inequality with Dirichlet boundary conditions on the ball \(B_{1}\) and since \(W\) is bounded from below on \(B_1 \setminus B_{1/2}\), we have by Weierstrass' theorem, that
\begin{equation*}
\begin{split}
\int_{B_1} \abs{v}^2  &= \int_{B_1} \psi^2\abs{v}^2+(1-\psi^2)\abs{v}^2
\leq \Cl{cPoincDirichlet} \int_{B_1} |\nabla(\psi v)|^2+\int_{B_1 \setminus B_{1/2}} (1-\psi^2)\abs{v}^2\\
&\leq 2\Cr{cPoincDirichlet}  \int_{B_{1}} \abs{\nabla v}^2+(2\Cr{cPoincDirichlet}\norm{\nabla \psi}_{\infty}^2+1 )\int_{B_{1}\setminus  B_{1/2}} \abs{v}^2
\leq \C \int_{B_1} \abs{\nabla v}^2+W\abs{v}^2.\qedhere
\end{split}
\end{equation*}
\end{proof}

We are now in position to prove Proposition~\ref{propositionLowerBound}.
\begin{proof}[Proof of Proposition~\ref{propositionLowerBound}]
\resetconstant
By taking if necessary a subsequence, we can assume that
\[
 \varliminf_{n \to \infty} \frac{1}{\varepsilon_n^{\kappa (N + \frac{p \gamma - \alpha}{p - 1})}} \mathcal{I}_{\varepsilon_n} (u_n)=
 \varlimsup_{n \to \infty} \frac{1}{\varepsilon_n^{\kappa (N + \frac{p \gamma - \alpha}{p - 1})}} \mathcal{I}_{\varepsilon_n} (u_n) < + \infty.
\]
We also observe that for each \(n \in \Nset\),
\begin{equation}
\label{normIdentity}
 \int_{\Rset^N} \varepsilon_n^2 \abs{\nabla u_n}^2 + V \abs{u_n}^2
 = \frac{2 p}{p - 1} \mathcal{I}_{\varepsilon_n} (u_n).
\end{equation}

By the scaled version of the classical Sobolev embedding theorem, we have, for each $q>1$ with $\frac{1}{q}\in(\frac{1}{2}-\frac{1}{N},\frac{1}{2})$, that for every \(x \in \Rset^N\)
\begin{equation*}
  \bigg(\int_{B_{\varepsilon_n^{\kappa} (x)}} \abs{u_n}^q \bigg)^\frac{2}{q}
    \leq \Cl{cUppBoundSobolevBall} \varepsilon_n^{\kappa N(\frac{2}{q} - 1)} \int_{B_{\varepsilon_n^{\kappa}} (x)} \varepsilon_n^{2\kappa} |\nabla u_n|^2+\abs{u_n}^2
\end{equation*}
here the Sobolev embedding constant $\Cr{cUppBoundSobolevBall}$ is independent of the point $x \in \Rset^N$, which, together with Lemma~\ref{lemglobal}, implies that
\begin{equation*}
 \bigg(\int_{B_{\varepsilon_n^{\kappa} (x)}} \abs{u_n}^q \bigg)^\frac{2}{q}
 \leq \Cl{cUppBoundSobolevBallV}\, \varepsilon_n^{\kappa(\frac{2 N }{q} - \gamma - N)}  \int_{B_{\varepsilon_n^{\kappa}} (x)} \varepsilon_n^2 \abs{\nabla u_n}^2+ V \abs{u_n}^2
\end{equation*}
and then
\begin{equation}
\label{eqglobalineq}
 \int_{B_{\varepsilon_n^{\kappa}} (x)}\abs{u_n}^q
 \leq \Cr{cUppBoundSobolevBallV}\, \varepsilon_n^{\kappa(\frac{2 N }{q} - \gamma - N)} \bigg(\int_{B_{\varepsilon_n^{\kappa}}(x)}\abs{u_n}^q \bigg)^{1- \frac{2}{q}} \int_{B_{\varepsilon_n^{\kappa}} (x)} \varepsilon_n^2 \abs{\nabla u_n}^2+ V \abs{u_n}^2.
\end{equation}
By integration both sides  on \eqref{eqglobalineq} and by Fubini's theorem we conclude that
\begin{equation*}
 \int_{\mathbb{R}^N}\abs{u_n}^q
 \leq \Cr{cUppBoundSobolevBallV}\,\varepsilon_n^{\kappa( \frac{2 N  }{q} - \gamma - N)}
 \bigg(\sup\limits_{x\in\mathbb{R}^N}\int_{B_{\varepsilon_n^{\kappa}}(x)}\abs{u_n}^q \bigg)^{1-\frac{2}{q}} \int_{\mathbb{R}^N} \varepsilon_n^2 \abs{\nabla u_n}^2+ V \abs{u_n}^2,
\end{equation*}
the constant $\Cr{cUppBoundSobolevBallV}$ depends neither on the point $x \in \Rset^N$ nor on the parameter $\varepsilon_n > 0$ provided that
$\varepsilon_n$ is small enough.
Since by assumption for every \(n \in \Nset\) the function \(u_n\) is a solution of the Choquard equation ($\mathcal{C}_{\varepsilon_n}$), we deduce from the Hardy--Littlewood--Sobolev inequality \eqref{HLS} that
\begin{equation*}
\begin{split}
\int_{\mathbb{R}^N}\varepsilon_n^2 & \abs{\nabla u_n}^2+ V \abs{u_n}^{2}\\
&=\int_{\mathbb{R}^N} (I_\alpha \ast \abs{u_n}^p)\abs{u_n}^p
\leq \C   \bigg(\int_{\mathbb{R}^N}\abs{u_n}^{\frac{2Np}{N+\alpha}}\bigg)^{\frac{N+\alpha}{N}}\\
&\leq \C\Biggl( \varepsilon_n^{\kappa(\frac{N + \alpha}{p} - \gamma - N)}
\biggl(\sup\limits_{x\in\mathbb{R}^N}\int_{B_{\varepsilon_n^{\kappa}} (x)}\abs{u_n}^{\frac{2Np}{N+\alpha}}
\biggr)^{1-\frac{N+\alpha}{Np}}\int_{\mathbb{R}^N}\varepsilon_n^2 |\nabla u_n|^2
+ V \abs{u_n}^2\Biggr)^{\frac{N+\alpha}{N}}
\end{split}
\end{equation*}
by the boundedness assumption on the sequence and by \eqref{normIdentity}, we then arrive at
\begin{equation*}
\varliminf_{n\to \infty}
\sup\limits_{x\in\mathbb{R}^N}
\frac{1}{\varepsilon_n^{\kappa (N + \frac{N p}{N + \alpha} \frac{\gamma - \alpha}{p - 1})}}
\int_{B_{\varepsilon_n^{\kappa}} (x)}\abs{u_n}^{\frac{2Np}{N+\alpha}}>0.
\end{equation*}
Hence, there exists a sequence of points $(x_n)_{n\in\mathbb{N}}\) in the space \(\mathbb{R}^N$ such that
\begin{equation}
\label{eqLowerBallsNpNalpha}
\varliminf_{n\to \infty}
\frac{1}{\varepsilon_n^{\kappa (N + \frac{N p}{N + \alpha} \frac{\gamma - \alpha}{p - 1})}}
\int_{B_{\varepsilon_n^\kappa} (x_n)}\abs{u_n}^{\frac{2Np}{N+\alpha}}>0.
\end{equation}
Since \(\frac{N p}{N + \alpha} - 1 >0\), \(1 - \frac{(N - 2)p}{N + \alpha}>0\),
\begin{gather}
\label{eqGagliardoNirenbergCondition1}
\Bigl(\frac{N p}{N + \alpha} - 1\Bigr) + \Bigl(1 - \frac{(N - 2)p}{N + \alpha}\Bigr)  = \frac{2 p}{N + \alpha}
\intertext{and }
\label{eqGagliardoNirenbergCondition2}
\Bigl(\frac{N p}{N + \alpha} - 1\Bigr) \Bigl(1 - \frac{2}{N}\Bigl) + \Bigl(1 - \frac{(N - 2)p}{N + \alpha}\Bigr)  = \frac{2}{N},
\end{gather}
by a scaling of the endpoint Gagliardo--Nirenberg interpolation inequality on
the ball and by Lemma~\ref{lemglobal}, we have
\begin{equation}
\label{eqLowerBalls2}
\begin{split}
&\int_{B_{\varepsilon_n^\kappa} (x_n)} \abs{u_n}^\frac{2Np}{N + \alpha}\\
 &\qquad \le \C \biggl(\int_{B_{\varepsilon_n^\kappa}(x_n)} \abs{\nabla u_n}^2 + \varepsilon_n^{-2 \kappa} \abs{u_n}^2\biggr)^{\frac{N}{2} (\frac{N p}{N + \alpha} - 1)} \biggl(\int_{B_{\varepsilon_n^\kappa}(x_n)} \abs{u_n}^2 \biggr)^{\frac{N}{2}(1 - \frac{(N - 2)p}{N + \alpha})}\\
 &\qquad\le \C \biggl(\frac{1}{\varepsilon_n^2} \int_{B_{\varepsilon_n^\kappa}(x_n)} \varepsilon_n^2 \abs{\nabla u_n}^2 + V \abs{u_n}^2\biggr)^{\frac{N}{2} (\frac{N p}{N + \alpha} - 1)}
 \biggl(\int_{B_{\varepsilon_n^\kappa}(x_n)} \abs{u_n}^2 \biggr)^{\frac{N}{2}(1 - \frac{(N - 2)p}{N + \alpha})}.
\end{split}
\end{equation}
Thus, by \eqref{normIdentity}, by the boundedness assumption on the energy and  by \eqref{eqLowerBallsNpNalpha},
we deduce from \eqref{eqLowerBalls2} that
$$
\varepsilon_{n}^{\kappa(N+\frac{Np}{N+\alpha}\frac{\gamma-\alpha}{p-1})}\leq \C \, \varepsilon_{n}^{\kappa(N-2+\frac{\gamma-\alpha}{p-1})\frac{N}{2}(\frac{Np}{N+\alpha}-1)}
\bigg(\int_{B_{\varepsilon^{\kappa}_n}(x_n)}\abs{u_n}^2\bigg)^{\frac{N}{2}(1-\frac{(N-2)p}{N+\alpha})}
$$
and we then have in view of the identities \eqref{eqGagliardoNirenbergCondition1} and \eqref{eqGagliardoNirenbergCondition2}  that
\begin{equation}
\label{eqLowerBalls2V}
 \varliminf_{n\to \infty}
\frac{1}{\varepsilon_n^{\kappa (N +  \frac{\gamma - \alpha}{p - 1})}}
\int_{B_{\varepsilon_n^\kappa} (x_n)}\abs{u_n}^2 > 0.
\end{equation}

On the other hand we have
\begin{equation}
\label{eqUpperBalls2V}
 \varlimsup_{n \to \infty} \frac{1}{\varepsilon_n^{\kappa (N + \frac{p\gamma - \alpha}{p - 1})}}
 \int_{B_{\varepsilon_n^\kappa}(x_n)} V \abs{u_n}^2
 \le \varlimsup_{n \to \infty} \frac{1}{\varepsilon_n^{\kappa (N + \frac{p\gamma - \alpha}{p - 1})}} \int_{\Rset^N} \varepsilon_n^2 \abs{\nabla u_n}^2 + V \abs{u_n}^2
 <+\infty.
 \end{equation}
By combining \eqref{eqLowerBalls2V} and \eqref{eqUpperBalls2V}, we deduce that
\begin{equation}\label{eqboundedpara1}
 \varlimsup_{n \to \infty} \frac{1}{\varepsilon_n^{\kappa \gamma}} \inf_{B_{\varepsilon_n^\kappa} (x_n)} V
 < + \infty.
\end{equation}
We claim that there exists \(x_* \in \Rset^N\) such that
\(V (x_*) = 0\) and up to a subsequence, the sequence \((x_n)_{n \in \Nset}\) satisfies the condition that
\begin{equation}\label{eqboundedpara2}
 \varlimsup_{n\to\infty}\frac{|x_n-x_*|}{\varepsilon^{\kappa}_n}<+\infty.
\end{equation}
In fact, by \eqref{eqboundedpara1}, there is a sequence $y_n\in B_{\varepsilon_n^{\kappa}}(x_n)$ such that
\[
  \varlimsup_{n\to\infty}\frac{1}{\varepsilon_{n}^{\kappa\gamma}}V(y_n) < +\infty.
\]
By Lemma~\ref{lemmaLowerBoundVEverywhere}, this implies that 
\[
 \varlimsup_{n\to\infty}\frac{\min_{1 \le i \le k} (\abs{y_n  - a_i}^\kappa)}{\varepsilon_{n}^{\kappa\gamma}} < +\infty,
\]
where \(\{a_1, \dotsc, a_k\} = V^{-1}(\{0\})\).
Thus, up to a subsequence, there exists a point \(x_* \in \{a_1, \dotsc, a_k\}\) such that 
\[
 \varlimsup_{n\to\infty}\frac{\abs{y_n - x_*}}{\varepsilon_{n}^{\gamma}} < +\infty, 
\]
and thus 
\[
 \varlimsup_{n\to\infty}\frac{\abs{x_n - x_*}}{\varepsilon_{n}^{\gamma}} 
 \le 1 + \varlimsup_{n\to\infty}\frac{\abs{x_n - x_*}}{\varepsilon_{n}^{\gamma}}< +\infty.
\]
In particular by \eqref{eqLowerBallsNpNalpha}, there exists \(R_* > 0\) such that
\begin{equation}
\label{eqLowerBallsNpNalphaR}
 \varliminf_{n\to \infty}
\frac{1}{\varepsilon_n^{\kappa (N + \frac{N p}{N + \alpha} \frac{\gamma - \alpha}{p - 1})}}
\int_{B_{R_* \varepsilon_n^\kappa} (x_*)}\abs{u_n}^{\frac{2Np}{N+\alpha}}>0.
\end{equation}

We define now for each \(n \in \Nset\) the rescaled function \(v_n : \Rset^N \to \Rset\) for each \(y \in \Rset^N\) by
\[
 v_n (y) = \varepsilon_n^{\kappa \frac{\alpha - \gamma}{2(p - 1)}} u_n (x_* + \varepsilon_n^\kappa y).
\]
Let \(W \in C (\Rset^N)\) be a positive \(\gamma\)--homogeneous function such
that
\[
 \varliminf_{x \to x_*} \frac{V (x) - W (x - x_*)}{\abs{x - x_*}^\gamma} \ge 0.
\]
We observe that since \(W\) is positive, this is equivalent to having
\begin{equation}
\label{condVW}
 \varlimsup_{x \to x_*} \frac{W (x - x_*)}{V (x)} \le 1.
\end{equation}
We now compute for each \(R > 0\) and \(n \in \Nset\),
\[
 \int_{B_R} \abs{\nabla v_n}^2 + W \abs{v_n}^2
 \le \frac{1}{\varepsilon_n^{\kappa (N + \frac{p \gamma - \alpha}{\gamma - \alpha})}}
 \bigg(\sup_{\abs{x - x_*} \le R \varepsilon_n^\kappa} \frac{W (x - x_*)}{V (x)}\bigg)
 \int_{\Rset^N} \varepsilon_n^2 \abs{\nabla u_n}^2 + V \abs{u_n}^2,
\]
and thus in view of \eqref{condVW}, for every \(R > 0\),
\begin{equation}
\label{eqAsymptotRW}
 \varlimsup_{n \to \infty} \int_{B_R} \abs{\nabla v_n}^2 + W \abs{v_n}^2
 \le \lim_{n \to \infty}
 \frac{1}{\varepsilon_n^{\kappa (N + \frac{p \gamma - \alpha}{\gamma - \alpha})}}
 \int_{\Rset^N} \varepsilon_n^2 \abs{\nabla u_n}^2 + V \abs{u_n}^2 < + \infty.
\end{equation}
By Lemma~\ref{lem4.3}, the sequence \((v_n)_{n \in \Nset}\) is bounded in \(H^1 (B_R)\).
By weak compactness and by a diagonal argument, there exists a function \(v_* : \Rset^N \to \Rset\) such that for each \(R > 0\), one has \(v_* \in H^1 (B_R)\) and the sequence \((v_n)_{n \in \Nset}\) converges weakly to \(v_*\) in the space \(H^1 (B_R)\).

By the lower semicontinuity of the norm, by \eqref{eqAsymptotRW} and
by \eqref{normIdentity} and by the boundedness assumption, we have
\begin{equation}\label{globalsobolev}
\begin{split}
\int_{\mathbb{R}^N} \abs{\nabla v_*}^2 + W \abs{v_*}^2
  &=\lim_{R\to\infty} \int_{B_R}\abs{\nabla v_*}^2+W\abs{v_*}^2 \\
  &\leq \lim\limits_{R\to\infty}\varliminf\limits_{n\to \infty} \int_{B_R}\abs{\nabla v_n}^2+W \abs{v_n}^2\\
& \le \lim_{n \to \infty}
 \frac{1}{\varepsilon_n^{\kappa (N + \frac{p \gamma - \alpha}{\gamma - \alpha})}}
 \int_{\Rset^N} \varepsilon_n^2 \abs{\nabla u_n}^2 + V \abs{u_n}^2 < + \infty.
\end{split}
\end{equation}
Moreover, in view of Rellich's compact embedding theorem, \((v_n)_{n \in \Nset}\) converges
strongly to \(v_*\) in \(L^{\frac{2 N p}{N + \alpha}} (B_{R})\) and thus in view of \eqref{eqLowerBallsNpNalphaR},
\[
 \int_{B_{R_*}} \abs{v_*}^\frac{2 N p}{N + \alpha}
 = \lim_{n \to \infty}\int_{B_{R_*}} \abs{v_n}^\frac{2 N p}{N + \alpha}  > 0.
\]
We observe that for each \(n \in \Nset\), the function \(v_n\) satisfies the equation
\begin{equation}
\label{eqRescaled}
 -\Delta v_n + V_n v_n = (I_\alpha \ast \abs{v_n}^p) \abs{v_n}^{p - 2} v_n,
\end{equation}
where the rescaled potential \(V_n\) is defined for each \(y \in \Rset^N\) by
\[
  V_n (y) = \frac{1}{\varepsilon_n^{\kappa \gamma}} V \bigl(x_* + \varepsilon_n^\kappa y\bigr).
\]
In order to pass to the limit in \eqref{eqRescaled}, we consider a test function $\varphi\in C_{c}^{\infty}(\mathbb{R}^N)$. We first have by the weak convergence on balls
\begin{equation}
\label{eqLimitLaplacian}
 \lim_{n \to \infty} \int_{\Rset^N} \nabla v_n \cdot \nabla \varphi
 = \int_{\Rset^N} \nabla v_* \cdot \nabla \varphi.
\end{equation}
If we assume that \(\varphi \ge 0\), since for each \(n \in \Nset\) the function \(v_n\) is nonnegative, we deduce by Fatou's lemma that
\begin{equation}
\label{eqLimitPotential}   \varliminf_{n\to\infty}\int_{\mathbb{R}^N} V_n v_n \varphi \geq \int_{\mathbb{R}^N}W v_* \varphi .
\end{equation}
We finally study the Riesz potential term.
We take $R>0$ large enough such that $\supp\varphi\subset B_R$. Since $v_n\rightharpoonup v_*$ in $H^1(B_R)$, thus
we have, as \(n \to \infty\),
\begin{equation}\label{eq4}
\chi_{B_{2R}}\abs{v_n}^{p-2}v_n \to \chi_{B_{2R}}\abs{v_*}^{p-2}v_* \text{ in } L^q(\mathbb{R}^N) \text{ with } \frac{1}{q}>\big(p-1\big)\Big(\frac{1}{2}-\frac{1}{N}\Big),
\end{equation}
where $\chi_{B_R}$ denotes the characteristic function of the ball $B_R$.
By the Hardy--Littlewood--Sobolev inequality \eqref{HLS}, we know that, as \(n \to \infty\),
\begin{equation}\label{eq5}
I_\alpha \ast (\chi_{B_{2R}}\abs{v_n}^{p}) \to I_\alpha \ast (\chi_{B_{2R}}\abs{v_*}^{p}) \text{ in } L^q(\mathbb{R}^N) \text{ with } \frac{1}{q}>p\Big(\frac{1}{2}-\frac{1}{N}\Big)-\frac{\alpha}{N}
\end{equation}
Moreover, since the sequence \((\abs{v_n}^p)_{n \in \Nset}\) converges weakly to \(\abs{v_*}^p\) in \(L^\frac{2 N}{N + \alpha} (\Rset^N)\) we have \cite{MorozVanSchaftingen2015}*{Proposition 3.4}, as \(n \to \infty\),
\begin{equation}\label{eq6}
I_\alpha \ast \big((1-\chi_{B_{2R}})\abs{v_n}^{p}\big) \to I_\alpha \ast \big((1-\chi_{B_{2R}})\abs{v_*}^{p}\big) \text{ in } L^\infty(B_R)
\end{equation}
Summarizing \eqref{eq5} and \eqref{eq6}, we obtain that, as \(n \to \infty\),
\[
I_\alpha \ast \abs{v_n}^{p} \to I_\alpha \ast \abs{v_*}^{p}  \text{ in } L^q(B_R) \text{ with } \frac{1}{q}>p\Big(\frac{1}{2}-\frac{1}{N}\Big)-\frac{\alpha}{N}
\]
In view of \eqref{eq4}, we get that
\[
(I_\alpha \ast \abs{v_n}^{p})\abs{v_n}^{p-2}v_n \to (I_\alpha \ast \abs{v_*}^{p})\abs{v_*}^{p-2}v_*  \text{ in } L^q(B_R),
\]
as \(n \to \infty\),
with
\[
\frac{1}{q}>p\Big(\frac{1}{2}-\frac{1}{N}\Big)-\frac{\alpha}{N}+(p-1)\Big(\frac{1}{2}-\frac{1}{N}\Big).
\]
Since \(\supp \varphi \subset B_R\), we have
\begin{equation}
\label{eqLimitRiesz}
  \lim_{n \to \infty} \int_{\mathbb{R}^N}(I_\alpha \ast \abs{v_n}^p)\abs{v_n}^{p-2}v_n \varphi = \int_{\mathbb{R}^N}(I_\alpha * \abs{v_*}^p)\abs{v_*}^{p-2}v_* \varphi.
\end{equation}

By the equation \eqref{eqRescaled}, and by the limits \eqref{eqLimitLaplacian}, \eqref{eqLimitPotential} and \eqref{eqLimitRiesz}, we have for every \(\varphi \in C^\infty_c (\Rset^N)\) with \(\varphi\geq 0\) that,
\begin{equation}
\label{eqLimitInequation}
  \int_{\Rset^N}\nabla v_*\cdot \nabla \varphi + W v_* \varphi \le
  \int_{\Rset^N}\bigl(I_\alpha \ast \abs{v_*}^p\bigr) \abs{v_*}^{p - 2} v_*\varphi.
\end{equation}
By \eqref{globalsobolev}, \(v_*\) is an admissible test function and thus
\[
 \int_{\Rset^N} \abs{\nabla v_*}^2 + W\abs{v_*}^2
 \le \int_{\Rset^N} \bigl(I_\alpha \ast \abs{v_*}^p\bigr) \abs{v_*}^p.
\]
There exists thus \(t_* \in (0, 1]\) such that
\[
 t_*^2 \int_{\Rset^N} \abs{\nabla v_*}^2 + W\abs{v_*}^2
 = t_*^{2 p} \int_{\Rset^N} \bigl(I_\alpha \ast \abs{v_*}^p\bigr) \abs{v_*}^p
\]
We have then
\[
\begin{split}
 \mathcal{E} (W)
 &\le \mathcal{J}_W (t_* v_*)
 = \Bigl(\frac{1}{2} - \frac{1}{2 p}\Bigr)t_*^{2p}
 \int_{\Rset^N} \bigl(I_\alpha \ast \abs{v_*}^p\bigr) \abs{v_*}^p\\
 &\le \Bigl(\frac{1}{2} - \frac{1}{2 p}\Bigr) \varliminf_{n \to \infty} \int_{\Rset^N} \bigl(I_\alpha \ast \abs{v_n}^p\bigr) \abs{v_n}^p
 \le \varliminf_{n \to \infty} \frac{1}{\varepsilon_n^{\kappa (N + \frac{p \gamma - \alpha}{p - 1})}}
 \mathcal{I}_{\varepsilon_n} (u_n).
\end{split}
\]

In the case where there exists a positive \(\gamma\)--homogeneous function \(W \in C (\Rset^N)\) such that
\[
 \lim_{x \to x_*} \frac{V (x) - W (x - x_*)}{\abs{x - x_*}^\gamma} = 0,
\]
we observe that equality holds in \eqref{eqLimitPotential} and thus in \eqref{eqLimitInequation}, so that the additional conclusion follows.
\end{proof}
\begin{proof}[Proof of Theorem~\ref{theoremGammaHomogeneous}]
This follows from Propositions~\ref{propositionExistence}, \ref{propositionUpperBound} and \ref{propositionLowerBound}.
In fact, we have
\[
  \mathcal{C} (x_*)
  = \underline{\mathcal{C}}(x_*)
  \leq\varliminf_{n\to\infty}
      \frac{c_{\varepsilon_n}}{\varepsilon_n^{\kappa(N+\frac{p\gamma-\alpha}{p-1})}}
  \leq \varlimsup_{n\to\infty}
        \frac{c_{\varepsilon_n}}
             {\varepsilon_n^{\kappa(N+\frac{p\gamma-\alpha}{p-1})}}
  \leq \inf_{x\in\Rset^N}\Bar{\mathcal{C}} (x)
  \leq  \mathcal{C} (x_*),
\]
where
$$
 c_{\varepsilon_n} = \mathcal{I}_{\varepsilon_n} (u_{\varepsilon_n})
 = \Bigl(\frac{1}{2} - \frac{1}{2 p}\Bigr) \int_{\Rset^N} \varepsilon_n^2 \abs{\nabla u_{\varepsilon_n}}^2 + V \abs{u_{\varepsilon_n}}^2.
$$
We thus deduce that
\begin{gather*}
 \mathcal{C} (x_*)
 = \mathcal{J}_{W_*} (v_*) =  \Bigl(\frac{1}{2} - \frac{1}{2 p}\Bigr) \int_{\Rset^N} \abs{\nabla v_*}^2 + W_* \abs{v_*}^2,
\end{gather*}
which yields the conclusion.
\end{proof}

\section{Asymptotics for a potential vanishing on an open set}

\label{sectionDeadCore}

This last section is devoted to the proof of Theorem~\ref{theoremDeadCore} which covers the case where the potential vanishes on the closure of smooth bounded open set.

\begin{proof}%
[Proof of Theorem~\ref{theoremDeadCore}]%
\resetconstant%
The existence of solutions for every \(\varepsilon \in (0, \varepsilon_0)\) follows immediately from Proposition~\ref{propositionExistence} with \(\varepsilon_0 > 0\).

We define the auxiliary functional \(\mathcal{K}_\varepsilon \in C^1 (H^1_V (\Rset^N))\)
for each \(v \in H^1_V (\Rset^N)\) by
\[
 \mathcal{K}_\varepsilon (v)
 = \frac{1}{2} \int_{\Rset^N} \abs{\nabla v}^2 + \frac{V}{\varepsilon^2} \abs{v}^2
 - \frac{1}{2 p} \int_{\Rset^N} \bigl(I_\alpha \ast \abs{v}^p\bigr) \abs{v}^p,
\]
and we observe that for every \(u \in H^1_V (\Rset^N)\),
\[
 \mathcal{K}_\varepsilon (\varepsilon^{-\frac{1}{p - 1}} u)
 = \varepsilon^{-\frac{2p}{p - 1}}
 \mathcal{I}_\varepsilon (u).
\]
Hence, we define for every \(\varepsilon > 0\), the function
\[
 v_{\varepsilon} = \varepsilon^{-\frac{1}{p - 1}} u_\varepsilon.
\]
We also consider the functional \(\mathcal{K}_* \in C^1 (E)\) defined for each \(v\in E\) by
\[
  \mathcal{K}_* (v)
 = \frac{1}{2} \int_{\Omega} \abs{\nabla v}^2
 - \frac{1}{2 p} \int_{\Omega} \bigl(I_\alpha \ast \abs{v}^p\bigr) \abs{v}^p,
\]
where
\[
  E = \bigl\{u \in H^1_V (\Rset^N) \st u = 0 \text{ in } \Rset^N \setminus \Omega\bigr \}.
\]
We observe that for every \(v \in E\), we have \(\mathcal{K}_* (v) = \mathcal{K}_\varepsilon (v)\), and thus, for every \(\varepsilon > 0\), since \(u_\varepsilon\) is a groundstate,
\[
\begin{split}
 \mathcal{K}_\varepsilon (v_\varepsilon)
 &= \inf\, \bigl\{ \mathcal{K}_\varepsilon (v) \st v \in H^1_V (\Rset^N)\setminus \{0\} \text{ and }
 \langle \mathcal{K}_\varepsilon' (v), v\rangle = 0 \bigr\}\\
 &\le c_* = \inf\, \bigl\{ \mathcal{K}_* (v) \st v \in E \setminus \{0\} \text{ and }
 \langle \mathcal{K}_*' (v), v\rangle = 0 \bigr\}.
\end{split}
\]
We deduce therefrom that for every \(\varepsilon > 0\), we have
\[
 \int_{\Rset^N} \abs{\nabla v_\varepsilon}^2 + \frac{V}{\varepsilon^2} \abs{v_\varepsilon}^2
 \le \frac{2 p}{p - 1} c_*.
\]
On the other hand, by Lemma~\ref{lem2.1}, we have if \(\varepsilon \le \varepsilon_0\),
\begin{equation}
\label{eqDeadCorePoincare}
 \int_{\Rset^N} \abs{\nabla v_\varepsilon}^2 +  \abs{v_\varepsilon}^2
 \le \Cl{cstxxns} \int_{\Rset^N} \abs{\nabla v_\varepsilon}^2 + \frac{V}{\varepsilon_0^2} \abs{v_\varepsilon}^2
 \le \Cr{cstxxns} \int_{\Rset^N} \abs{\nabla v_\varepsilon}^2 + \frac{V}{\varepsilon^2} \abs{v_\varepsilon}^2
\end{equation}
and thus
\[
 \limsup_{\varepsilon \to 0} \int_{\Rset^N} \abs{\nabla v_\varepsilon}^2 + \abs{v_\varepsilon}^2
 < + \infty.
\]
It follows that there exists a sequence \((\varepsilon_n)_{n \in \Nset}\) in \((0, +\infty)\) converging to \(0\) such that the sequence \((v_{\varepsilon_n})_{n \in \Nset}\) converges weakly in \(H^1 (\Rset^N)\) to some function \(v_* \in H^1 (\Rset^N)\).
By Rellich's theorem, this sequence also converges strongly in \(L^{2}_{\mathrm{loc}} (\Rset^N)\).

If the set \(U \subset \Rset^N\) is open and if \(\Bar{\Omega} \subset U\), then, since \(\varliminf_{\abs{x} \to \infty} V (x) > 0\), we have \(\inf_{\Rset^N \setminus U} V > 0\), and thus for every \(\varepsilon_n > 0\),
\[
 \int_{\Rset^N \setminus U} \abs{v_{\varepsilon_n}}^2
 \le \frac{\varepsilon_n^2}{\inf_{\Rset^N \setminus U} V} \int_{\Rset^N}
 \abs{\nabla v_{\varepsilon_n}}^2 + \frac{V}{\varepsilon_n^2} \abs{v_{\varepsilon_n}}^2,
\]
so that
\[
 \lim_{n\to\infty} \int_{\Rset^N \setminus U} \abs{v_{\varepsilon_n}}^2 = 0.
\]
It follows thus that \((v_{\varepsilon_n})_{n \in \Nset}\) converges strongly to \(v_*\) in \(L^2 (\Rset^N)\). By the Gagliardo--Nirenberg--Sobolev interpolation inequality
we have
\begin{multline*}
  \int_{\Rset^N} \abs{v_{\varepsilon_n} - v_*}^\frac{2 N p}{N + \alpha}\\
  \le \C \Bigl(\int_{\Rset^N} \abs{\nabla (v_{\varepsilon_n} - v_*)}^2 + \abs{v_{\varepsilon_n} - v_*}^2 \Bigr)^{\frac{N}{2}(\frac{N p}{N+\alpha} - 1)}
  \Bigl(\int_{\Rset^N} \abs{v_{\varepsilon_n} - v_*}^2\Bigr)^{\frac{N}{2}(1 - \frac{(N-2) p}{N+\alpha})},
\end{multline*}
so that in view of \eqref{eqDeadCorePoincare}, the sequence \((v_{\varepsilon_n})_{n \in \Nset}\) converges also strongly to \(v_*\) in \(L^\frac{2 N p}{N + \alpha} (\Rset^N)\).
Moreover we also have \(v_* = 0\) on \(\Rset^N \setminus \Bar{\Omega}\).

In view of the Hardy--Littlewood--Sobolev inequality \eqref{HLS}, the classical Sobolev inequality and of \eqref{eqDeadCorePoincare}, we have, for each \(n \in \Nset\),
\[
\begin{split}
\int_{\Rset^N} \abs{v_{\varepsilon_n}}^\frac{2 N p}{N + \alpha}
 &\le \C \Bigl( \int_{\Rset^N}\abs{\nabla v_{\varepsilon_n}}^2 + \abs{v_{\varepsilon_n}}^2 \Bigr)^\frac{N p}{N + \alpha}
 \le \Cl{uxrxjn} \Bigl(\int_{\Rset^N}
          \abs{\nabla v_{\varepsilon_n}}^2 + \frac{V}{\varepsilon_n^2} \abs{v_{\varepsilon_n}}^2
          \Bigr)^\frac{N p}{N + \alpha}\\
 &=
 \Cr{uxrxjn} \Bigl(\int_{\Rset^N} \bigl(I_\alpha \ast \abs{v_{\varepsilon_n}}^p\bigr)
        \abs{v_{\varepsilon_n}}^p\Bigr)^\frac{N p}{N + \alpha}
  \le \C  \Bigl(\int_{\Rset^N} \abs{v_{\varepsilon_n}}^\frac{2 N p}{N + \alpha} \Bigr)^p.
\end{split}
\]
Since \(p > 1\), we deduce that
\[
\int_{\Rset^N} \abs{v_*}^\frac{2 N p}{N + \alpha}
=\lim_{n \to \infty} \int_{\Rset^N} \abs{v_{\varepsilon_n}}^\frac{2 N p}{N + \alpha}> 0,
\]
and thus \(v_*\ne 0\).

If we consider now a test function \(w \in E\), we have for every \(n \in \Nset\),
\begin{equation*}
\begin{split}
 0 &= \int_{\Rset^N} \nabla v_{\varepsilon_n} \cdot \nabla w
 + \frac{V}{\varepsilon_n^2} v_{\varepsilon_n} w
 - \int_{\Rset^N} \bigl(I_\alpha \ast \abs{v_{\varepsilon_n}}^p\bigr) \abs{v_{\varepsilon_n}}^{p - 2} v_{\varepsilon_n} w\\
 &= \int_{\Omega}  \nabla v_{\varepsilon_n}\cdot \nabla w
 - \int_{\Omega} \bigl(I_\alpha \ast \abs{v_{\varepsilon_n}}^p\bigr) \abs{v_{\varepsilon_n}}^{p - 2} v_{\varepsilon_n} w.
\end{split}
\end{equation*}
By the weak convergence of the sequence \((v_{\varepsilon_n})_{n \in \Nset}\) in \(H^1 (\Rset^N)\), by its strong convergence in \(L^\frac{2N p}{N + \alpha} (\Rset^N)\) and by Hardy--Littlewood--Sobolev inequality \eqref{HLS}, we deduce that
\[
\begin{split}
\int_{\Omega}  \nabla v_* \cdot \nabla w
 &- \int_{\Omega} \bigl(I_\alpha \ast \abs{v_*}^p\bigr) \abs{v_*}^{p - 2} v_* w\\
&= \lim_{n\to\infty}\bigg(\int_{\Omega}  \nabla v_{\varepsilon_n} \cdot \nabla w
 - \int_{\Omega} \bigl(I_\alpha \ast \abs{v_{\varepsilon_n}}^p\bigr) \abs{v_{\varepsilon_n}}^{p - 2} v_{\varepsilon_n} w\bigg)=0.
\end{split}
\]
In view of the regularity assumptions on the set \(\Omega\) and classical regularity theory, the function \(v_*\) satisfies the announced equation.

We also have
\[
\begin{split}
 \limsup_{n \to \infty} \int_{\Rset^N} \abs{\nabla v_{\varepsilon_n}}^2
 &\le \lim_{n \to \infty} \int_{\Rset^N} \abs{\nabla v_{\varepsilon_n}}^2 + \frac{V}{\varepsilon_n^2} \abs{v_{\varepsilon_n}}^2
 = \lim_{n \to \infty} \int_{\Rset^N} \bigl(I_\alpha \ast \abs{v_{\varepsilon_n}}^p\bigr) \abs{v_{\varepsilon_n}}^p \\
 &= \int_{\Rset^N} \bigl(I_\alpha \ast \abs{v_*}^p\bigr) \abs{v_*}^p
 = \int_{\Rset^N} \abs{\nabla v_*}^2.
\end{split}
\]
This implies that \((v_{\varepsilon_n})_{n \in \Nset}\) converges strongly in \(H^1 (\Rset^N)\) to \(v_*\) and that \(v_*\) is a groundstate of the limiting equation.
\end{proof}

\end{document}